\documentclass[11pt,leqno]{amsart}

\usepackage{epsfig}
\usepackage[utf8]{inputenc}

  \usepackage{amsmath}
\usepackage{amsthm}
\usepackage{mathtools}
\mathtoolsset{showonlyrefs}
\usepackage{xcolor}

\usepackage[T1]{fontenc}

\usepackage{amssymb}

\newtheorem{theorem}{Theorem}

\newtheorem{proposition}{Proposition}

\newtheorem{lemma}{Lemma}
\newtheorem{remark}{Remark}

\newcommand{\eps}{\varepsilon}

\newcommand{\C}{\mathcal{C}}

\newcommand{\D}{\mathcal{D}}

\newcommand{\M}{\mathcal{M}}

\newcommand{\abs}[1]{\left\vert#1\right\vert}

\newcommand\CC{\hbox{C\kern -.58em {\raise .54ex \hbox
			{$\scriptscriptstyle |$}}
		\kern-.55em {\raise .53ex \hbox{$\scriptscriptstyle |$}} }}
\newcommand\qd{\hfill$\sqcap\kern-8.0pt\hbox{$\sqcup$}$}
\newcommand\NN{\hbox{I\kern-.2em\hbox{N}}}
\newcommand\nn{\hbox{I\kern-.2em\hbox{N}}}
\newcommand\RR{I\!\!R}
\newcommand\sRR{{\sl \hbox{I\kern-.2em\hbox{R}}}}
\newcommand\QQ{\hbox{I\kern-.53em\hbox{Q}}}
\newcommand\sign{\hbox{Sign}}

\everymath{\displaystyle}

\newcommand{\cs}{$^\dagger$} \newcommand{\cm}{$^\ddagger$}

\usepackage{stackengine}
\usepackage{scalerel}
\usepackage{xcolor,amssymb}

\begin{document}
 
\title{$BV$-Estimates for Non-Linear Parabolic PDE \\  with Linear Drift}

\author[E. Erraji]{El Mahdi Erraji\cs} 

\author[N. Igbida]{Noureddine Igbida\cm} \thanks{\cm Institut de recherche XLIM-DMI, UMR-CNRS 6172, Facult\'e des Sciences et Techniques, Universit\'e de Limoges, France. Emails:   noureddine.igbida@unilim.fr}

\author[F. Karami]{Fahd Karami\cs} 

\author[D. Meskine]{Driss Meskine\cs}\thanks{\cs  Laboratoire de Mathématique Informatique et modélisation des Systèmes Complèxes, EST Essaouira,  Cadi Ayyad University 
	Marrakesh,   Morocco.   Emails:  el.erraji@uca.ac.ma, fa.karami@uca.ac.ma, dr.meskine@uca.ac.ma.}

\date{\today}
%****************************************************************************
 
 \maketitle

\begin{abstract}
In the present work, we establish space Bounded Variation $(BV)$ regularity of the solution  for a  non-linear parabolic partial differential equations involving a linear drift term.    We study the problem  in a  bounded domain with  mixed Dirichlet-Neumann boundary conditions, a general non-linearity and reasonable  assumptions on the data.    Our results  also cover, as a particular case, the linear transport equation in a bounded domain with an outward-pointing drift vector field on the boundary. 
\end{abstract}

\section{Introduction}
 Over the last few decades, many works have focused their attention to    the degenerate  parabolic PDE  involving a   linear drift term 
\begin{equation} \label{EvolEq00}
	 	\left\{ 
\begin{array}{l}  	\displaystyle \frac{\partial u }{\partial t}  -\Delta p +\nabla \cdot (u  \: V)=f  \\  \\  u\in \beta (p) 
\end{array}\right. \quad   \hbox{ in } Q:= (0,T)\times \Omega .
\end{equation}
 Our aim here is to establish a local bounded variation   (BV for short) estimate for solutions     subject to {Dirichlet-Neumann} mixed boundary condition
%\begin{equation}
%	\left\{ \begin{array}{ll}
%	\displaystyle p= 0  & \hbox{ on }\Sigma_D := (0,T)\times \Gamma_D\\  \\  
%\displaystyle (\nabla p-u\: V)\cdot \nu= 0  & \hbox{ on }\Sigma_N := (0,T)\times \Gamma_N,
% \end{array} \right.
%\end{equation} 
and initial data. 
 Here   $\Omega\subset\RR^N$ is a bounded   open set with regular boundary $\displaystyle \partial\Omega =:\Gamma $,  the graph  $\beta$ is maximal monotone and $V\: :\: \Omega\to \RR^N$ is a given vector field.
 
 \medskip 
 The prototypical   equation \eqref{EvolEq00} {has} attracted 
significant attention in recent years due to their ability to model 
phenomena in which transport and nonlinear diffusion interact. 
Such equations arise naturally in population dynamics, biology, 
porous-medium flow, and phase-transition processes.  Typical examples include the porous medium 
equation $\beta(r)=r^m$ (see, e.g., \cite{Vbook}), {crowd} motion where $\beta(r)=\mathrm{sign}(r)$ (see, e.g. \cite{MRS1,IgCrowed}), and 
Stefan-type phase-transition problems.

The analytic study of \eqref{EvolEq00} is rich and challenging. 
Existence of weak solutions can be obtained by several complementary 
methods. Wasserstein variational methods may be useful for the case of  Neumann boundary conditions (cf. \cite{Santambrogio15}). Nonlinear semigroup theory may be applicable when   monotonicity is achievable (cf. \cite{IgPME}). Duality variational methods (minimal flux approach), particularly effective for the general case involving mixed and non-homogeneous boundary conditions, as demonstrated in \cite{IgMF}. 

In contrast to existence, uniqueness remains largely open in the general 
case. Classical comparison arguments ensure uniqueness only when 
$\beta$ is sufficiently regular. For singular or multivalued 
nonlinearities, the situation is far more delicate. Even in the 
porous-medium or Hele--Shaw regimes, uniqueness may require more involved techniques (renormalization à la Diperna-Lions, doubling and dedoubling variables techniques) and  additional 
structural assumptions  (see, for example, 
\cite{IgShaw,IgPME,DiMe,Noemi}).

Further analytical difficulties arise when studying for instance the asymptotic 
behaviour of solutions and their dependence on physical parameters. 
Such questions typically require strong compactness of either the 
density $u$ or the so-called pressure $p$. However, classical energy estimates 
provide weak compactness, namely $u$ in $L^1(Q)$ and $p$ in $L^2(0,T;H^1(\Omega))$, but this 
is insufficient to pass to the limit in nonlinear terms, due to the  degeneracy of the diffusion operator. In this context, 
space \emph{$BV$-estimates} for $u$ play a central role, as they supply the  missing strong compactness required for stability and asymptotic 
analysis. For Dirichlet boundary conditions, space $BV$-type estimates for the case of  constant vector {field} $V$ were obtained in \cite{Ca}, other results may be found in the \cite{DeMeSan} but much less is known in the present framework.
 
 \medskip
 The aim of this work is to establish local space  $BV$-estimates for solutions 
of~\eqref{EvolEq00} under reasonable  
assumptions on the data and on the nonlinearity $\beta$. Our results  also cover, as a particular case, the linear transport equation  obtained by taking $\beta^{-1} \equiv 0$.

\medskip  
\textbf{Plan of the paper.}  In the following section, Section 2, we introduce the main assumptions and state the primary results of this paper. To establish these results, we employ an implicit Euler time discretization scheme. Our analysis {begins} in Section 3, where we prove the crucial estimate for the solution to the associated stationary problem. Section 4 is then devoted to the detailed proofs of the main results concerning the evolution problem.

\section{Main results}
\setcounter{equation}{0}
Let us consider the evolution problem, which will serve as the dynamical framework for our analysis, namely
 
\begin{equation}  \label{Pevol}
	\left\{  \begin{array}{ll} 
		\left. 
		\begin{array}{l}  	\displaystyle \frac{\partial u }{\partial t}  -\Delta p +\nabla \cdot (u  \: V)=f  \\  \\  u\in \beta (p) 
		\end{array}\right\} \quad  & \hbox{ in } Q:= (0,T)\times \Omega\\  \\  
		\displaystyle p= 0  & \hbox{ on }\Sigma_D := (0,T)\times \Gamma_D\\  \\  
		\displaystyle (\nabla p-u\: V)\cdot \nu= 0  & \hbox{ on }\Sigma_N := (0,T)\times \Gamma_N\\  \\      
		\displaystyle  u (0)=u _0 &\hbox{ in }\Omega,\end{array} \right.
\end{equation} 
where   $\nu$ represents the   outward unitary normal to  the boundary $\Gamma,$ $\beta$ is a given maximal monotone graph, $u_0\in L^2(\Omega)$ and  $f\in L^2(Q).$  We assume throughout the analysis that the vector field V satisfies the following set of assumptions: 
 \begin{itemize}
	\item[$(T1)$] Regularity: $V\in L^2(0,T;H^1(\Omega)^N)$ 
	\item[$(T2)$]  Divergence control: $(\nabla \cdot V)^-\in L^\infty(Q)  $ 
		\item[$(T3)$]  Boundary conditions:   $  	V\cdot \nu  \geq 0,$ on   $\Gamma_D,$ and   $  	V\cdot \nu  = 0,$ on   $\Gamma_N.$
\end{itemize}

\medskip 
To study the problem \eqref{Pevol}, we use the {associated} $\eps-$Euler implicit {scheme}, for    an arbitrary $0<\eps\leq \eps_0$ and $n\in \NN^*$ being  such that  $(n+1)\eps=T.$  We consider the sequence of stationary  {problems} given by  :
\begin{equation}	\label{sti}
	\left\{  \begin{array}{ll}\left.
		\begin{array}{l}
			\displaystyle u_{i+1} - \eps \: \Delta p_{i+1}  +  \eps \: \nabla \cdot (u_{i+1}   \: V_i)=u_{i} +\eps  \: f_i  \\
			\displaystyle u_{i+1} \in \beta(p_{i+1}) \end{array}\right\}
		\quad  & \hbox{ in }  \Omega\\  \\
		\displaystyle p_{i+1}= 0  & \hbox{ on }  \Gamma_D\\    \\  
		\displaystyle (\nabla p_{i+1}-u_{i+1}\: V_i)\cdot \nu = 0  & \hbox{ on }  \Gamma_N ,\end{array} \right. \quad i=0,1, ... n-1,
\end{equation}
where,   for each $i=0,...n-1,$ $f_i$ and $V_i$ are  given by
$$f_i = \frac{1}{\eps } \int_{i\eps}^{(i+1)\eps }  f(s)\: ds, \quad V_i = \frac{1}{\eps } \int_{i\eps}^{(i+1)\eps }  V(s)\: ds,\quad \hbox{ a.e. in }\Omega.  $$
Here the integration of the vector field $V$  is to be interpreted on a component-by-component basis.

%Proposition
 \begin{proposition}\label{Peuler}
		Under the assumptions $(T1)$ and $(T2),$ for any $u_0\in L^2(\Omega) $ and $f\in L^2(Q),$  the  sequence of couple  $(u_i,p_i)$  solving the problems \eqref{sti},    for $i={0,\ldots,n-1},$ is well defined ; in the sense that   $u_i \in L^1(\Omega),$  $p_i\in H^1_{\Gamma_D}(\Omega)$ and 
		\begin{equation} 
			\displaystyle  \int_\Omega u_{i+1}\:\xi+   \eps\: \int_\Omega  \nabla {p_{i+1}}\cdot  \nabla\xi  - 
			 \eps\:	\int_\Omega  u_{i+1} \:  V_i\cdot \nabla  \xi    =     \eps\: \int_\Omega f_{i}\: \xi +  \int_\Omega u_i\:\xi  , 
		\end{equation} 
		for any $	\xi\in H^1_{\Gamma_D}(\Omega).$ 
		 	\end{proposition}

 \bigskip 
 Now, for a given $\eps-$time discretization $0=t_0<t_1<t_2<...<t_n<t_{n+1}= T,$ satisfying
$t_{i+1}-t_i \leq \eps,  $  we define the $\eps-$approximate solution by
\begin{equation}\label{epsapprox}
	u_\eps:=  \sum_{i=0} ^{n-1 }  u_{i+1}\chi_{[t_i,t_{i+1})},\quad 	p_\eps:=   \sum_{i=0} ^{n-1 }  p_{i+1}\chi_{[t_i,t_{i+1})}  ,\quad 	f_\eps:=   \sum_{i=0} ^{n-1 }  f_i\chi_{[t_i,t_{i+1})}    , 
\end{equation}
 and      \begin{equation} 	V_\eps:=   \sum_{i=0} ^{n-1 }  V_i\chi_{[t_i,t_{i+1})}
\end{equation} 
See that, taking  
$$\tilde u_\eps(t) =\frac{(t-t_i)u_{i+1} -  (t-t_{i+1})u_i}{\eps}  \quad \hbox{ for any }t\in [t_i,t_{i+1}),\:  i={0,\ldots,n-1},$$
we have 
\begin{equation}\label{weakeps0}
	\partial_t\tilde u_\eps -\Delta p_\eps +\nabla \cdot(u_\eps\: V_\eps)=f_\eps,\quad \hbox{ in }\D'(Q).
\end{equation}
A solution of the evolution problem \eqref{Pevol} may be given by   letting $\eps\to 0$  in \eqref{weakeps0}.  This {is} the aim of the following theorem.

 %Theorem 
\begin{theorem}\label{texistevolm}
	Under the assumptions $(T1)$,  $(T2)$ and $(T3),$ for any $u_0\in L^2(\Omega) $ and $f\in L^2(Q),$
	there exists $u  \in  L^\infty(0,T;L^{2}(\Omega))$ and $p\in L^2(0,T;H^1_{\Gamma_D}(\Omega)),$ such that, by taking a sub-sequence if necessary as $\eps\to 0,$ we have
	\begin{equation}\label{convueps}
		u_\eps \to u \quad   \hbox{ in } L^2(Q)-\hbox{weak},  
	\end{equation} 
	\begin{equation}\label{convtildeueps}\tilde u_\eps \to u \quad   \hbox{ in } L^2(Q)-\hbox{weak}
	\end{equation} 
	and \begin{equation}\label{convpeps}
		p_\eps \to p \quad \hbox{ in }L^2(0,T;H^1_{\Gamma_D}(\Omega))-\hbox{weak} \quad \hbox{ as }\eps\to 0.
	\end{equation}    
The couple $(u,p)$ is a weak solution of \eqref{Pevol} in the sense that   $u  \in  L^\infty(0,T;L^{2}(\Omega)),$  $\partial_t u \in  L^2(0,T; W^{-1,p'}_{\Gamma_D}(\Omega)) ,$ $u(0)=u_0$  in $\Omega,$  $p \in  L^2(0,T;H^{1}_{\Gamma_D}(\Omega)),$   $u\in   \beta( p)$ a.e. in $Q$  and  
	\begin{equation}  \label{evolweak}
		\frac{d}{dt}\int_\Omega u\: \xi\: dx + \int_\Omega   (\nabla p   -u\: V)  \cdot \nabla \xi \: dx =\int_{\Omega} f\: \xi\: dx,\quad \hbox{ in }\D'([0,T)) \quad \mbox{for any }\; \xi\in W^{1,2}_{\Gamma_D}(\Omega).
	\end{equation}
	Moreover, we have 
	\begin{enumerate}
 \item For any $q\in [1,\infty],$ we have  
		\begin{equation} \label{lquevol} 
			\Vert u(t)\Vert_{L^q(\Omega)}    \leq  \M_q(t):=  \left\{ 
			\begin{array}{lll}
				e^{ (q-1)\:   \int_0^t \Vert (\nabla \cdot V)^-\Vert_{L^\infty(\Omega)}  }  \left( \Vert u_0 \Vert_{L^q(\Omega)} +   \int_0^t   \Vert f (t)\Vert_{L^q(\Omega)} \: dt  \right ) \quad &\hbox{ if } & q<\infty  \\ 
				e^{ \int_0^t  \Vert (\nabla \cdot V)^-\Vert_{L^\infty(\Omega)}  }  \left( \Vert u_0 \Vert_{L^\infty(\Omega)} +   \int_0^t   \Vert f (t)\Vert_{L^\infty(\Omega)} \: dt  \right)&\hbox{ if } & q=\infty.  
			\end{array}\right.  . 
		\end{equation}
		
		\item  For any $t\in [0,T),$ we have 
		\begin{equation}\label{lmuevol}
			\begin{array}{c} 
				\frac{d}{dt} \int_\Omega j(u) \: dx+  \int_\Omega \vert \nabla p\vert^2  \: dx\leq     \int_\Omega f\:p \: dx +       \int _{\Omega}    p \: u  \:  (\nabla \cdot  V )^-    \: dx \quad \hbox{ in }\D'(0,T).
			\end{array} 
		\end{equation} 	       
		
		 	\item If $u_0\geq 0 $ and $f\geq 0,$ then $u\geq 0.$
		
	\end{enumerate}
\end{theorem} 

\begin{remark} One sees that under the assumption of Theorem \ref{texistevolm}, we can find a constant $ C(N,\Omega)$ depending only on Poincaré inequality constant, such that 
	\begin{equation} \label{hestimatem} 	 
		\begin{array}{cc}
			\frac{d}{dt} \int_\Omega j( u) \: dx+  \int_\Omega \vert \nabla p\vert^2\: dx    \leq    C (N,\Omega)   \left(  	\int_\Omega \vert f\vert^2\: dx    +	 \Vert (\nabla \cdot V)^-  \Vert_{L^\infty(\Omega)}\: \M_2^2      \right ) \quad \hbox{ in }\D'(0,T). 
		\end{array}	
	\end{equation}   
\end{remark}

\medskip 

\begin{remark}The general uniqueness of the solution stated in Theorem \ref{texistevolm} remains an open and challenging problem. This difficulty persists even in cases where the function $\beta$ is continuous. For the particular settings of the Porous Medium Equation $(\beta(r)=r^\alpha),$ and the Hele-Shaw problem $(\beta(r)=\hbox{Sign}(r)),$  we refer the reader to the works \cite{IgPME,IgShaw,DiMe,Noemi} and the discussions contained therein.

\end{remark}

To set our main result concerning $BV-$estimate on the solution of \eqref{Pevol}, let us remind the reader some basic definitions and tools for  functions of bounded variation.  A given function   $u\in L^1(\Omega)$ is said to be of bounded variation if and only if, for each $i=1,...N,$    
\begin{equation}
	TV_i(u,\Omega)  := \sup\left\{ 	\int_\Omega  u\: \partial_{x_i} \xi \: dx  \:  :\:  \xi\in \mathcal C^1_c(\Omega)  \hbox{ and }\Vert \xi\Vert_{L^\infty(\Omega)} \leq 1\right\}  <\infty, \end{equation}
here   $\C^1_c(\Omega)$ denotes the set of $\C^1-$function compactly supported in $\Omega.$   More generally a function is locally of bounded variation in a domain $\Omega$ if and only if for any open set $\omega\subset\!   \subset \Omega,$   $ TV_i(u,\omega)<\infty$ for any $i=1,...,N.$ 
In general a function  locally of bounded variation (as well as function of bounded variation) in $\Omega,$  may not be differentiable, but by the Riesz representation theorem, their  partial derivatives  in the sense of distributions are  Borel measure in $\Omega.$  This gives rise to the definition of the   vector space of functions of   bounded variation in $\Omega$, usually denoted by    $BV(\Omega),$ as the set of   $u\in L^1(\Omega)$ for which there   are Radon measures $\mu_1,...,\mu_N$ with finite total mass in $\Omega$ such that   
\begin{equation}
	\int_\Omega v\: \partial_{x_i} \xi \: dx =-\int_\Omega \xi\: d\mu_i,\quad \hbox{ for any }\xi\in \mathcal C_{c}(\Omega),\quad \hbox{ for }i=1,...,N. 
\end{equation}
Without abuse of notation  we continue to point out the measures $\mu_i$ by $\partial_{x_i}v$ anyway,  and by $\vert \partial_{x_i}v\vert $   the total variation of $\mu_i.$ Moreover, we'll use  as usual $Dv=(\partial_{x_1}v,...,\partial_{x_N}v)$  the vector valued Radon measure  pointing out the gradient of any function $v\in BV(\Omega),$ and $\vert Dv\vert$ indicates the total variation measure of  $v.$   In particular, for any open set $\omega\subset\!   \subset \Omega,$ $TV_i(v,\omega)=\vert \partial_{x_i}v\vert(\omega)<\infty,$ and the total variation of the function $v$ in  $\omega$  is finite too ; i.e. 
\begin{equation}
	\vert Dv\vert(\omega)= \sup\left\{ 	\int_\Omega  v\: \nabla \cdot \xi \: dx  \:  :\:  \xi\in \mathcal C^1_c(\omega)^N  \hbox{ and }\Vert \xi\Vert_{L^\infty(\Omega)} \leq 1\right\}  <\infty.
\end{equation}
At last, let us remind the reader here  the well known  compactness result for functions of bounded variation : given a sequence $v_n$ of functions in $BV_{loc}(\Omega)$ such that, for any   open set $\omega\subset\!   \subset  \Omega,$ we have 
\begin{equation}
	\sup_n\left\{ \int_\omega \vert v_n\vert\: dx + \vert Dv_n\vert (\omega) \right\} <\infty,
\end{equation} 
there exists  a subsequence that we denote again by $v_n$  which converges in $L^1_{loc}(\Omega)$ to a function   $v\in BV_{loc}(\Omega).$  Moreover, for any  compactly supported continuous function $0\leq \xi$, the limit $v$ satisfies 
\begin{equation}
	\int \xi\:  \vert \partial_{x_i}v \vert \leq \liminf_{n\to\infty }  \int \xi\: \vert \partial_{x_i}v_n \vert ,  
\end{equation}
for any $i=1,...N,$ and 
\begin{equation}
	\int \xi\:  \vert Dv\vert \leq \liminf_{n\to\infty }  \int \xi\:  \vert Dv_n\vert.
\end{equation}

\medskip
Now, we can set our main result concerning space $BV_{loc}$-estimates. 

%Theorem 
\begin{theorem}\label{tbvlocuevol}
Assume that the vector field $V$ satisfies  moreover  \begin{itemize}
		\item[$(T'1)$] $V\in L^1(0,T;W^{1,\infty}(\Omega)^N) \cap L^\infty(Q) $  
		\item[$(T'2)$]  $\nabla \cdot V\in L^2(0,T;W^{1,2 }_{loc}(\Omega)) \cap L^\infty(Q)$
		\item[$(T'3)$]    $  	V\cdot \nu  = 0,$ on   $\Gamma_N,$ and moreover     
		 \begin{equation}\label{HypBV}
		 		V(x)\cdot \nu(\pi(x)) \geq 0,\quad \hbox{ for a.e.  }x\in \Omega\setminus \Omega_h, 
		 \end{equation}
	for any $0<h<h_0$ small enough{,} where      $$\Omega_h=\Big\{ x\in \Omega\: :\: d(x,\Gamma)>h \Big\},$$ and      $\pi(x)$  denotes the projection of $x$ on the boundary $\Gamma.$ 
	\end{itemize} 
For a given $f\in L^1(0,T;BV_{loc}(\Omega))\cap L^2(Q)$ and    $u_0\in BV_{loc}(\Omega)\cap L^2(\Omega),$   
	 let us consider $(u,p)$ the solution given by Theorem \ref{texistevolm}. Then, for  any  open set $\omega \subset\!\subset \Omega,$ there exists    $0\leq \xi \in \mathcal C^2_c(\Omega)$ such that   $\xi\chi_{\omega}\equiv 1,$ and there exists a constant $ C:= C(\omega,N,\Omega,  \Vert f\Vert_{L^2}, \Vert (\nabla \cdot V)^-\Vert _\infty{)}$    such that  
	\begin{equation}\label{bvlocuevol} 
		\begin{array}{cc}
			\sum_{j=1}^N    \vert \partial_{x_j} u(t) \vert (\omega)	 \leq C\left( 1+ \int_0^T\!\! \left(\int_\Omega \vert f\vert^2\: dx +	 \Vert (\nabla \cdot V)^-  \Vert_{L^\infty(\Omega)}^2\: \M_2(T)^2  \right)   +  	\sum_{j=1}^N     \int_\Omega   \vert  \partial_{x_j}(\nabla \cdot  V {)} \vert^2 \: dx     \right. \\ 
			\hspace*{1cm}	\left.     +	 \sum_{j=1}^N    	\int_0^T  \vert \partial_{x_j} f(t) \vert(\Omega)+  \sum_{j=1}^N    \vert \partial_{x_j} u_0 \vert(\Omega) \right), \quad \hbox{ for any }t\in (0,T).
		\end{array}
	\end{equation}  
\end{theorem} 
%%%%%%%%

\begin{remark}
\begin{enumerate}
	\item 
 See that the condition $(T'3)$  implies definitely $(T3).$   But the converse part is not necessary true.    In fact, with $(T'3)$ , we are assuming  that  $V$ is outpointing along the paths given by the distance function in a neighborhood of $\Gamma.$  As we will see, this assumption can be weaken into an outpointing vector field condition  along a given arbitrary paths in the neighborhood of $\Gamma$ (cf. Remark \ref{Rbvcond}).    Anyway, controlling the outward-pointing orientation of $V$ near the boundary appears crucial for managing the oscillations of the solution (at least for the stationary problem)   and obtaining local   $BV$-estimates.  
 
 \item 
 
 We do not know whether the result of the Theorem holds true in the absence of condition $(T'3).$ However, this condition appears to be   linked to the second-order term and its associate  Neumann boundary condition ;  $\nabla p\cdot \nu =0.$ 
 As we shall demonstrate later, this condition is unnecessary whenever the graph $\beta$ vanishes and the PDE is reduced to a first-order linear conservation  equation with an outpointing vector field on $\partial \Omega.$

 \item In the case where $f\equiv 0$ and $\nabla \cdot V =0,$ we retrieve a more or less well known result 
	\begin{equation} 
	\begin{array}{cc}
		\sum_{j=1}^N    \vert \partial_{x_j} u(t) \vert (\omega)	 \leq    \sum_{j=1}^N    \vert \partial_{x_j} u_0 \vert(\Omega)  , \quad \hbox{ for any }t\in (0,T).
	\end{array}
\end{equation}

\end{enumerate}	 
\end{remark}

\bigskip 
Now, let us focus for a while on the continuity  equation    
\begin{equation}\label{PDEtransport0}
	\left\{ 
	\begin{array}{ll} 
		\displaystyle \frac{\partial  u }{\partial t}  + \nabla\cdot (   u\: V )= f\quad & \hbox{ in } Q\\  \\   
		u(0)= u_0 & \hbox{ in }\Omega, 
	\end{array} 
	\right.
\end{equation} 
where   $Q:=(0,T)\times \Omega,$ with $\Omega$ an open  bounded domain.  Following   \cite{Donadello1,Donadello2}, to incorporate the boundary condition for \eqref{PDEtransport0}, we need to define the inflow boundary as :
$$\Sigma^- := \Big\{x\in \partial \Omega \: :\:  V\cdot \nu <0\Big\},$$ 
and consider the evolution problem \eqref{PDEtransport0} with homogeneous Dirichlet boundary condition. That is  
\begin{equation} \label{PDEtransportBC}
	\left\{ 
	\begin{array}{ll} 
		\displaystyle \frac{\partial  u }{\partial t}  + \nabla\cdot (   u\: V )=  f \quad & \hbox{ in } Q\\  \\   
		u   =0 & \hbox{ on }\Sigma^-:= (0,T)\times\Sigma^- \\  \\  
		u(0)= u_0 & \hbox{ in }\Omega. 
	\end{array} 
	\right.
\end{equation}    
 Under the assumptions $(T1)$ and $(T2)$, we know that (cf. \cite{Donadello1,Donadello2}) 
 for any   $  u_0\in L^\infty(\Omega)$ and $f\in L^\infty(Q), $    
	the problem \eqref{PDEtransportBC}    has a unique   weak solution $ u$ in the sense that 	:   $ u\in L^\infty(Q)$,  $( u\: V)\cdot \nu \in L^\infty(\Sigma) $, $( u\: V)\cdot \nu   =0 ,$ $\mathcal H^ {N-1}-$a.e. on $\Sigma^- ,$ and  
	\begin{equation}\label{weakF0}
		\frac{d}{dt} 	\int_\Omega    u\: \xi \: dx -\int_\Omega  u\: V\cdot \nabla \xi \: dx =  \int_\Omega f\:\xi \: dx , \hbox{ in } {\D'([0,T))},
	\end{equation}	 
	for any $\xi \in \C^\infty_c(\Omega).$

	As a consequence, assuming that 
	\begin{equation}\label{Vassump}
		V\cdot \nu \geq 0 ,\quad \hbox{ on } \partial \Omega,
	\end{equation}
clearly $\Sigma^-=\emptyset,$ and then the problem \eqref{PDEtransport0} admits a unique solution.  In particular, $( u\: V)\cdot \nu   $ is uniquely well defined on $\Sigma.$   	As a consequence of Theorem \ref{tbvlocuevol}, under the assumption \eqref{Vassump}, we can prove the following $BV$-regularity results for the solution $u.$

%Theorem 
 \begin{theorem}\label{TBVreg}
	Assume that $V\in L^\infty(0,T;W^{1,\infty}(\Omega)^N)  \cap L^2(0,T; W^{2,2}_{loc}(\Omega)^N) .$ For any  
 $f\in L^1(0,T;BV (\Omega))\cap L^2(Q)$ and $u_0\in BV(\Omega)\cap L^2(\Omega),$  the unique solution $u$ of \eqref{PDEtransport0} satisfies 
 
 \begin{enumerate}
 \item 	 $u\in L^1(0,T; BV(\Omega) )$ 
 \item  there exists a constant $ C:= C(N,\Omega,  \Vert f\Vert_{L^2(Q)}, \Vert (\nabla \cdot V)^-\Vert_{L^\infty(Q)}{)}$,  such that  
	\begin{equation} 
		\begin{array}{cc}
			\sum_{j=1}^N    \vert \partial_{x_j} u(t) \vert (\Omega)	 \leq C\left( 1+ \int_0^T\!\! \left(\int_\Omega \vert f\vert^2\: dx + \M_2(T)^2 	\:  \Vert (\nabla \cdot V)^-  \Vert_{L^\infty(\Omega)}^2  \right)   +  	\sum_{j=1}^N     \int_\Omega   \vert  \partial_{x_j}(\nabla \cdot  V {)} \vert^2 \: dx     \right. \\ 
			\hspace*{1cm}	\left.     +	 \sum_{j=1}^N    	\int_0^T  \vert \partial_{x_j} f(t) \vert(\Omega)+  \sum_{j=1}^N    \vert \partial_{x_j} u_0 \vert(\Omega) \right) \quad \hbox{ for any }t\in (0,T).
		\end{array}
	\end{equation} 
	\end{enumerate} 
	 	\end{theorem}

\bigskip 
To prove this theorem we consider   $\tilde V$  {to} be such that $\tilde V\in L^1(0,T;W^{1,\infty}(\RR^N)^N) \cap L^\infty((0,T)\times \RR^N) ^N$,   $\nabla \cdot V\in L^2(0,T;W^{1,2 }_{loc}(\RR^N)) \cap L^\infty((0,T)\times \RR^N)$ and there exists a bounded open set $\Omega \subset \! \subset D\subset \RR^N$  such that  
$$\hbox{support}(\tilde V(t)) \subset \! \subset  D, \quad \hbox{ for a.e. }t\in (0,T) $$
and 
$$ 	\tilde V(t,x) = V(t,x)\qquad\text{for all }x\in \Omega,\ \text{a.e. }t\in (0,T).$$
 More precisely, $\tilde V$ can be constructed as follows:
 	\begin{enumerate}
		\item For a.e.\ $t$, we extend $V(t,\cdot)$ to a function $EV(t,\cdot)$ in 
		$W^{1,\infty}(\mathbb{R}^N)^N$ using a standard bounded extension operator for 
		Lipschitz domains.
		
		\item Choose $\varphi\in C_c^\infty(\mathbb{R}^N)$ such that 
		$\varphi\equiv 1$ on $\overline{\Omega}$.
		
		\item Define
	$\tilde V(t,x) := \varphi(x)\, EV(t,x).$
	\end{enumerate}
	 	Then $\tilde V$ satisfies
$	\tilde V \in L^1(0,T; W^{1,\infty}(\mathbb{R}^N)^N)
	\cap L^\infty((0,T)\times\mathbb{R}^N),$ 
	$$ 
	\hbox{supp}(\tilde V(t,\cdot)) \subseteq \hbox{supp}(\varphi) =:D\Subset \mathbb{R}^N,$$ 
	and coincides with $V$ on $\overline \Omega$.
 Moreover, since 
 $$\nabla\!\cdot\tilde V 
 = \varphi\, \nabla\!\cdot(EV) + \nabla\varphi\cdot EV,$$
 the assumption $V\in L^2(0,T; W^{2,2}_{loc}(\Omega)^N)$ implies that $\nabla \cdot \tilde V \in L^2(0,T; W^{1,2}_{loc}(\Omega)).$ 
 
 \medskip 
 This being said, let us consider the following continuity equation:
\begin{equation} \label{tildetransport}
	\left\{
	\begin{array}{ll}
		\displaystyle \frac{\partial  u }{\partial t}  + \nabla\cdot (   u\: \tilde V )= \tilde f\quad & \hbox{ in } (0,T)\times D\\
		u(0)= \tilde u_0 & \hbox{ in }{D},
	\end{array}
	\right.
\end{equation}
Here, $\tilde f$ and $\tilde u_0$ {are} given by a standard bounded extension operator for Lipschitz domains, compactly supported in $D$.
  Since the extended vector field $\tilde V$ is compactly supported in $D$, its normal component must vanish on the boundary:
$$
\tilde V\cdot \nu =0 \quad \hbox{ on }\partial D.
$$
Consequently, the inflow boundary associated with $\tilde V$ in $D$, defined as $\tilde \Sigma^-:= (0,T)\times(\partial D)^-$, is empty ($\tilde \Sigma^- = \emptyset$).  Thanks to \cite{Donadello1}, the problem \eqref{tildetransport} admits a  unique weak solution, $\tilde u$.  In particular $\tilde u$ satisfies the weak formulation \eqref{weakF0}. By uniqueness, $\tilde u$ defines an extension of $u.$ That is  
$$\tilde u=u \quad \hbox{ a.e. in }Q.  $$
Moreover, by choosing the maximal monotone graph $\beta^{-1} \equiv 0$, one can see that $\tilde u$  is also a solution of the problem  \eqref{Pevol}  defined in $D$:
 \begin{equation}
	\left\{ \begin{array}{ll}
		\left.
		\begin{array}{l}
			\displaystyle \frac{\partial \tilde u }{\partial t}  -\Delta \tilde p +\nabla \cdot (\tilde u  \: \tilde V)=\tilde f  \\
			{\tilde u\in \beta (\tilde p)}
		\end{array}\right\} \quad  & \hbox{ in }  (0,T)\times D  \\
		\displaystyle \tilde p= 0  & \hbox{ on }  (0,T)\times \partial D\\
		\displaystyle  \tilde u (0)=\tilde u _0 &\hbox{ in }D.
	\end{array} \right.
\end{equation} 
Since $\tilde V$ is compactly supported, we deduce by Theorem \eqref{tbvlocuevol}  that $\tilde u\in L^1(0,T;BV_{loc}(D){)}$ and there exists  
{a} constant $ C:= C(\omega,N,D,  \Vert \tilde f\Vert_{L^2((0,T)\times D)}, \Vert (\nabla \cdot \tilde V)^-\Vert _{L^\infty((0,T)\times D)}{)}$    such that  
\begin{equation} 
	\begin{array}{cc}
		\sum_{j=1}^N    \vert \partial_{x_j} \tilde u(t) \vert (\omega)	 \leq C\left( 1+ \int_0^T\!\! \left (\int_D \vert \tilde f\vert^2\: dx +	 \Vert (\nabla \cdot \tilde V)^-  \Vert_{L^\infty(D)}^2  \right)   +  	\sum_{j=1}^N     \int_D   \vert  \partial_{x_j}(\nabla \cdot  \tilde V {)} \vert^2 \: dx     \right. \\ 
		\hspace*{1cm}	\left.     +	 \sum_{j=1}^N    	\int_0^T  \vert \partial_{x_j} \tilde f(t) \vert(D)+  \sum_{j=1}^N    \vert \partial_{x_j} \tilde u_0 \vert(D) \right) \quad \hbox{ for any }t\in (0,T).
	\end{array}
\end{equation}  
 Taking $\omega=\Omega,$ we deduce the result of the theorem.

	\begin{remark} 
		Without the assumption  $V\in L^2(0,T; W^{2,2}_{loc}(\Omega)^N)$, the above construction does  not  in general preserve the additional regularity
		$$	\nabla\!\cdot \tilde V \in L^2(0,T; W^{1,2}_{loc}(\RR^N)).$$
		Indeed, even though $EV$ belongs to $W^{1,\infty}(\mathbb{R}^N)^N$, this does not guarantee
		that the divergence $\nabla\!\cdot(EV)$ belongs to 
		$W^{1,2}_{loc}(\mathbb{R}^N)$, because this would require control of the 
		\emph{second derivatives} of $V$. Such control is not available under the sole 
		assumption $V \in W^{1,\infty}(\Omega)$.
		Moreover, when multiplying by a cutoff function $\varphi$, the divergence becomes
		$\nabla\!\cdot(\varphi\, EV)
		= \varphi\, \nabla\!\cdot(EV) + \nabla\varphi\cdot EV,$
		where the term $\nabla\varphi\cdot EV$ typically fails to belong to 
		$W^{1,2}_{loc}$, unless one imposes additional second-order regularity on $V$, 
		such as $V \in L^2(0,T; W^{2,2}_{loc}(\Omega)^N)$. 	Therefore, the extension preserves the space 
		$L^1(0,T; W^{1,\infty}(\Omega)) \cap L^\infty((0,T)\times \Omega)^N $, 
		but does \emph{not} preserve in general the condition 
		$\nabla\!\cdot V \in L^2(0,T; W^{1,2}_{loc}(\Omega))$.
	\end{remark}

\section{Stationary problem}
 \setcounter{equation}{0}

Thanks to the Euler implicit discretization \eqref{sti}, our analysis starts with the generic stationary problem  
\begin{equation}	\label{st}
	\left\{  \begin{array}{ll} 
		\displaystyle v -\lambda\Delta p + \lambda   \nabla \cdot (v  \: V)=f,  
		\quad  v\in \beta(p),\quad & \hbox{ in }  \Omega\\    \\  
		\displaystyle p= 0  & \hbox{ on }  \Gamma_D\\    \\  
		\displaystyle (\nabla p-v\: V)\cdot \nu = 0  & \hbox{ on }  \Gamma_N ,\end{array} \right.
\end{equation}
where $\displaystyle f \in L^2(\Omega)$, $\lambda>0.$   
A  function   $v \in L^1(\Omega)$ is said to be a weak solution of \eqref{st}  if  $p\in H^1_{\Gamma_D}(\Omega)$ and 
\begin{equation}\label{stwf}
	\displaystyle  \int_\Omega v\:\xi+  \lambda \int_\Omega  \nabla p \cdot  \nabla\xi  - 
	\lambda	\int_\Omega  v \:  V\cdot \nabla  \xi    =     \int_\Omega f\: \xi, \quad \hbox{ for all }	\xi\in H^1_{\Gamma_D}(\Omega).
\end{equation}

%% Theorem 
\begin{proposition} \label{texistm}
	Assume $V\in W^{1,2}(\Omega),$     $(\nabla \cdot V)^-\in L^\infty(\Omega) $ and satisfies $(T3)$.    For any $f\in L^2(\Omega)$ and $\lambda$, such that  
	\begin{equation}\label{condlambda1}
		0<\lambda <  \lambda_0:= 1/\Vert (\nabla\cdot  V)^-\Vert_{L^\infty(\Omega)} ,  
	\end{equation}
	the problem \eqref{st} has a solution  that we denote by $v.$  Moreover,  $v$ satisfies the following properties: 
	
	\begin{enumerate}
		\item  For any $1\leq q< \infty,$   we have 
	        \begin{equation} \label{lqstat}  
	 	\Big( 1-  (q-1)   \lambda \: \Vert (\nabla\cdot  V)^-\Vert_{L^\infty(\Omega)}    \Big)   \Vert v\Vert_{L^q(\Omega)} \leq    \Vert f\Vert_{L^q(\Omega)}, \quad 
	 		\end{equation}  
	 	 
              \item   \begin{equation}\label{lmst}
		\Big( 1-  \lambda \: \Vert (\nabla\cdot  V)^-\Vert_{L^\infty(\Omega)}    \Big ) \int_\Omega v\: p    \: dx + 	\lambda	\int_\Omega \vert \nabla  p \vert^2\: dx 	 \leq  \int_\Omega f\: p  \: dx    . 
	\end{equation}

	\item For any $k\in \RR,$ we have 
\begin{equation}\label{k-contraction}
	 	\int _\Omega(v-k)^+ \:  dx \leq   \int _\Omega (f- k(1-\lambda \: \Vert (\nabla \cdot V)^-\Vert_{L^\infty(\Omega)})) ^+ \: dx . 
\end{equation}
and \begin{equation}\label{k-contraction-}
	\int _\Omega(k-v)^+ \:  dx \leq   \int _\Omega (-f+ k(1-\lambda \: \Vert (\nabla \cdot V)^-\Vert_{L^\infty(\Omega)})) ^+ \: dx . 
\end{equation} 
	\end{enumerate}
	In particular, we have 
\begin{itemize}
	\item If $f\in L^\infty(\Omega),$ then $v\in L^\infty(\Omega)$ and 
		\begin{equation} \label{linftystat}   	\Big( 1-     \lambda \: \Vert (\nabla\cdot  V)^-\Vert_{L^\infty(\Omega)}    \Big)   \Vert v\Vert_{L^\infty(\Omega)}  		\leq  \Vert f\Vert_{L^\infty(\Omega)}.	\end{equation}  
		\item If $f\geq 0,$ then $v\geq 0.$ 
\end{itemize}
 
	%and, if  $f_1\leq f_2,$  a.e. in  $\Omega,$   then ws
	%$$u_1\leq u_2,\quad \hbox{ a.e. in  }\Omega.$$  
\end{proposition}
%%%%%%%%%%%%%%%%
% 
 To prove Proposition \ref{texistm},  we proceed by regularization and compactness argument. For each $\delta >0,$ we consider   $\beta_\delta $  a regular Lipschitz continuous  function strictly increasing satisfying $\beta_\delta(0)=0$ and, as $\delta\to0,$  
$$\beta_\delta \to \beta ,\quad \hbox{ in the sense of graph in }\RR\times \RR.  $$
One can take  for instance,  $\beta_\delta$ the Yosida approximation of the application $r\in \RR\to r^{1/m}.$  Then, we  consider the problem 
\begin{equation}	\label{pstbeta}
	\left\{  \begin{array}{ll}  \displaystyle v - \lambda\Delta p +  \lambda \nabla \cdot (v  \: V)=f , \quad   
		\displaystyle v =\beta_\delta   (p)  
		\quad  & \hbox{ in }  \Omega\\   \\  
	\displaystyle p= 0  & \hbox{ on }  \Gamma_D\\    \\  
	\displaystyle (\nabla p-v\: V)\cdot \nu = 0  & \hbox{ on }  \Gamma_N  .\end{array} \right.
\end{equation}

Let us define the functions $H_\sigma$ and $H^+_\sigma$ as smooth approximations of the mappings $\sign_0$ and $\sign_0^+$ respectively. They are given by

\begin{equation}\label{Hsigma}
		H_\sigma   (r) = \left\{ \begin{array}{ll} 
			1 \quad &  \hbox{ if } r \geq \sigma \\
			r/\sigma & \hbox{ if }  |r|\leq\sigma \\
			-1 &\hbox{ if } r \leq  -\sigma
				\end{array}\right.  
\qquad 
		H^+_\sigma   (r) = \left\{ \begin{array}{ll} 
			1 \quad &  \hbox{ if } r\geq \sigma \\
			r/\sigma & \hbox{ if }  0\leq r<\sigma \\
			0 &\hbox{ if }r\leq 0
			 \: .	\end{array}\right.  
	\end{equation} 
where $\sigma$ denotes the regularization parameter that controls the smoothness of the approximation.
%Lemma  
\begin{lemma}\label{lexistreg}
	For any $f\in L^2(\Omega)$, $\delta>0$ and  $\lambda$ satisfying \eqref{condlambda1},    the problem 	\eqref{pstbeta} has a     solution $v_\delta ,$ in the sense that  $v_\delta \in L^2(\Omega),$ $p_\delta:=\beta_\delta^{-1} (v_\delta) \in H^1_{\Gamma_D}(\Omega),$        and 
	\begin{equation}
		\label{weakeps}  
		\int_\Omega v_\delta \: \xi \: dx + \lambda \int_\Omega \nabla p_\delta \cdot \nabla \xi \: dx  -\lambda\int_\Omega v_\delta \: V\cdot \nabla \xi \: dx =\int_\Omega f\: \xi\: dx,   \end{equation}
	for any $\xi\in H^1_{\Gamma_D}(\Omega).$ Furthermore,  
	the solution $v_\delta$ satisfies all the properties of Proposition \ref{texistm}.
\end{lemma}  
\begin{proof}
	To simplify the presentation we omit the subscript  $\delta$ in the notations of $(v_\delta,p_\delta)$ and $\beta_\delta$ throughout this proof. We denote by  $H^{-1}_{\Gamma_D} (\Omega) $ the usual topological dual space of $H^1_{\Gamma_D}(\Omega)$ and $ \langle .,. \rangle$ the associate  duality pairing. Consider  $A\: :\:  H^1_{\Gamma_D}(\Omega)\to  H^{-1}_{\Gamma_D} (\Omega)$ defined  by 
	$$ \langle A p,\xi\rangle =  \int_\Omega\beta  (p)\: \xi\: dx + \lambda\: \int _\Omega\nabla p\cdot \nabla \xi \: dx -\lambda\: \int_\Omega\beta  (p) \: V\cdot \nabla \xi \: dx,\quad \hbox{ for any }\xi,\: p\in H^1_{\Gamma_D}(\Omega), $$
	is a bounded weakly continuous operator.  Moreover,    for any $\lambda$ satisfying \eqref{condlambda1},   $A$ is coercive. Indeed, for any $p\in H^1_{\Gamma_D}(\Omega),$ we have 
	\begin{equation} \label{coercive}
		\begin{array}{lll}	\langle A p,p\rangle &=& \int_\Omega\beta  (p)\: p \: dx + \lambda\: \int_\Omega \vert \nabla p\vert^2  \: dx -\lambda\:  \int_\Omega\beta  (p) \: V\cdot \nabla p  \: dx \\ 
			%	&=&  \int_\Omega\beta  (p)\: p \: dx  +   \lambda\: \int_\Omega \vert \nabla p\vert^2  \: dx -\lambda\: \int_\Omega   V\cdot \nabla\left(  \int_0^p  \beta  (r) dr\right)    \: dx \\   
			&=&  \int_\Omega\beta  (p)\: p \: dx  +  \lambda\:  \int_\Omega \vert \nabla p\vert^2  \: dx -\lambda\:  \int_\Omega   V\:   \cdot  \nabla \left(  \int_0^p  \beta  (r) dr\right)    \: dx  \\  
			&\geq&  \int_\Omega\beta  (p)\: p \: dx  +  \lambda\:  \int _\Omega\vert \nabla p\vert^2  \: dx -\lambda\:  \int _\Omega p  \beta  (p)   \:   \left(\nabla \cdot   V\right)^-   \:    dx    -  \lambda\: \underbrace{ \int _{\Gamma_N}     p  \beta  (p)  \:   V \cdot   \nu       \: dx}_{= 0}  \\ 
			&\geq&  \lambda\:   \int_\Omega \vert \nabla p\vert^2  \: dx +(1- \lambda \: \Vert (\nabla \cdot V)^-\Vert_{L^\infty(\Omega)} )\int_\Omega  \beta  (p)\: p \: dx  \\ 
			&\geq&  \lambda\:   \int_\Omega \vert \nabla p\vert^2  \: dx ,  
		\end{array}
	\end{equation}
	
	where we use the fact that $V\cdot \nu = 0$ on $\Gamma_N$ and $ 0\leq \int_0^ p\beta (r)dr\leq p\beta (p) .$ 
	So, for any $f\in  H^{-1}_{\Gamma_D} (\Omega)  $ the problem $A p=f$ has a solution $p \in H^1_{\Gamma_D}(\Omega).$   
	Now, for each  $1<q<\infty,$   taking   $   v ^{q-1}  $ as a test function in \eqref{weakeps}. Using the fact that 
	$v\nabla (  v ^{q-1}  )=   \frac{q-1}{q}\:  \nabla \vert v\vert^q  $  
	and $ \nabla p\cdot \nabla  (  v ^{q-1}  )  \geq 0  $ a.e. in $\Omega $  
	we get 
	\begin{eqnarray}
		\int_\Omega \vert v\vert^q  \: d x
		&\leq& \int_\Omega f\:   v ^{q-1} \: dx +\lambda \frac{q-1}{q}  \int _\Omega V\cdot  \nabla \vert v\vert ^q   \: dx 
		\\ 
		&\leq&  \int_\Omega f  v ^{q-1 }  \: dx   -  \lambda \frac{q-1}{q}  \int_\Omega  \nabla \cdot  V\:   \vert v\vert ^q   \: dx      +  \lambda \frac{q-1}{q} \underbrace{ \int_{\Gamma_N }   V\cdot \nu \:   \vert v\vert ^q   \: dx  }_{=0}	\\ 
		&\leq&   \int_\Omega f  v ^{q-1 }  \: dx   + \lambda \frac{q-1}{q}  \int _\Omega\left(  \nabla \cdot  V\right)^- \:   \vert v\vert ^q   \: dx         \\  
		&\leq& \frac{1}{q} \int_\Omega \vert f\vert^q   \: dx +    \frac{q-1}{q} \int _\Omega\vert v\vert^q   \: dx   + \lambda \frac{q-1}{q} \left\Vert  \left(  \nabla \cdot  V\right)^- \right\Vert_{L^\infty(\Omega)}   \int_\Omega \:   \vert v\vert ^q   \: dx, 
	\end{eqnarray}
	where we use     again   Young inequality. This implies      \eqref{lqstat} for $v_\eps$ in the case where $1<q<\infty.$ 
	The case $q=1$ follows clearly  by letting $q\to 1.$    	For the remaining case $q=  \infty ,$   we take  $H^+_\sigma (v-k)\in H^1_{\Gamma_D}(\Omega),$  for a given $k\geq 0$ and $\sigma>0,$     as a test function in  \eqref{weakeps}.	Then, letting $\sigma \to 0$ and using  the fact that $ \nabla p\cdot \nabla  H_\sigma^+  (v-k)\geq 0$ a.e. in $\Omega,$   we  see that 
	\begin{equation}
		\begin{array}{lll}
			\int_\Omega (v-k)^+ \:  dx & \leq&  \int_\Omega  (f- k(1+\lambda \: \nabla \cdot V)) \:\sign_0^+(v-k)  \: dx+\lambda \lim_{\sigma \to 0} 
			\int_\Omega (v-k)\: V\cdot \nabla H_\sigma^+ (v-k)  \: dx \\ 
			&\leq&  \int _\Omega (f- k(1+\lambda \: \nabla \cdot V)) \:\sign_0^+(v-k) \: dx,
	\end{array}  \end{equation} 
	where we use the fact that 
	$\lim_{\sigma \to 0} 
	\int_\Omega (v-k)\: V\cdot \nabla H_\sigma^+ (v-k)  \: dx=\lim_{\sigma \to 0} \frac{1}{\sigma}
	\int_{0\leq v-k\leq \sigma} (v-k)\: V\cdot \nabla (v-k) \:   \: dx= 0.$ Thus  \eqref{k-contraction}. \\ 
	Now,  taking 
	$	k=\frac{\Vert f\Vert_{L^\infty(\Omega)}}{1-\lambda\: \Vert (\nabla \cdot V)^-\Vert_{L^\infty(\Omega)} },$ 
	we deduce that  $v\leq \frac{\Vert f\Vert_{L^\infty(\Omega)}}{1-\lambda\: \Vert (\nabla \cdot V)^-\Vert_{L^\infty(\Omega)} }.$ Similarly, we apply the same argument with $H^+_\sigma (-v+k)$ as a test function, we obtain  
	$  	v\geq - \frac{\Vert f\Vert_{L^\infty(\Omega)}}{1-\lambda\: \Vert (\nabla \cdot V)^-\Vert_{L^\infty(\Omega)} }.  $ 
	This finished to proof of  \eqref{linftystat}.  
	To prove \eqref{lmst} for $p_\eps,$  we take  $p$ as a   in \eqref{weakeps} and work as in 
	\eqref{coercive},  
	we obtain 
	\begin{eqnarray}
		\lambda	\int _\Omega \vert \nabla p\vert^2\: dx &=&    \int_\Omega  fp\: dx -\int_\Omega vp\: dx -\lambda \int_\Omega \nabla\cdot V  \:    \left(  \int_0^p  \beta  (r) dr\right) \: dx \\ 
		&\leq& \int _\Omega fp\: dx -\int_\Omega vp\: dx +\lambda\: \Vert (\nabla \cdot V)^- \Vert_{L^\infty(\Omega)} \: \int_\Omega   vp  \: dx . 
	\end{eqnarray}
	Thus \eqref{lmst}.  At last, assuming $f\geq 0,$ one sees easily that taking $k=0$ in \eqref{k-contraction-} we get $v\geq 0.$

\end{proof}

%Lemma
\begin{lemma}\label{lconveps}
	Under the assumption of Proposition \ref{texistm}, by taking a sub-sequence $\delta\to 0$ if necessary, we have 
	\begin{equation}\label{weakueps}
		v_\delta\to v \quad \hbox{ in }L^2(\Omega)\hbox{-weak}
	\end{equation}  
	and 
	\begin{equation}\label{strongpeps}
		p_\delta\to p \quad \hbox{ in }H^1_{\Gamma_D}(\Omega). 
	\end{equation}    
	Moreover, $v$ is a weak solution of \eqref{st}. 
\end{lemma}
%%%%%%%%%
\begin{proof}  Using Lemma \ref{lexistreg} as well as Young and Poincar\'e inequalities, we see that the  sequences  $v_\delta$ and $p_\delta $  are bounded  in  $L^2(\Omega)$ and $H^1_{\Gamma_D}(\Omega),$ respectively.  So, there exists a subsequence that we denote again by $v_\delta$ and $p_\delta $   such that \eqref{weakueps} is fulfilled,    
	\begin{equation}\label{weakpeps}
	 v_\delta\to v \quad \hbox{ in }L^2(\Omega)\hbox{-weak},   \quad	p_\delta\to p\quad \hbox{ in }H^1_{\Gamma_D}(\Omega)\hbox{-weak} 
	\end{equation}   
	and $v\in \beta(p) .$   	Letting $\delta\to 0$ in \eqref{weakeps}, we obtain that $v$ is a weak solution of \eqref{st}. Let us  prove that actually \eqref{weakpeps} holds to be true strongly in $H^1_{\Gamma_D}(\Omega).$ Indeed,  taking $p_\delta$ as a test function, we get  
	\begin{equation}  
		\begin{array}{ll} 
			\lambda \int _\Omega\vert \nabla p_\delta\vert^2\: dx  &=   \int_\Omega   (f- v_\delta )\: p_\delta\: dx +\lambda  \int_\Omega    V \cdot \nabla  \left(\int_0^{p_\delta}  \beta_\delta(r)dr\right) \: dx\\ 
			&= \int_\Omega   (f- v_\delta )\: p_\delta\: dx -\lambda    \int _\Omega     \nabla \cdot   V   \:     \int_0^{p_\delta}  \beta_\delta(r)dr  \: dx    . 
		\end{array}
	\end{equation}     
	Since $\int_0^r \beta_\delta(s)\: ds$ converges to  $\int_0^r \beta(s)\: ds,$ for any $r\in \RR,$ $p_\delta\to p$ a.e. in $\Omega$ and $\left\vert  \int_0^{p_\delta} \beta_\delta(s)\: ds\right\vert \leq v_\delta\: p_\delta$ which is bounded in $L^1(\Omega)$ by \eqref{lmst}, we have  
	$$    \int_0^{p_\delta} \beta_\delta(s)\: ds \to     \int_0^{p} \beta_0(s)\: ds ,\quad \hbox{ in } L^1(\Omega). $$ 
	So, in one hand we have  
	\begin{equation} 
		\begin{array}{ll}  	\lim_{\delta\to 0}  
			\lambda 	\int_\Omega \vert \nabla p_\delta\vert^2\: dx  &=  \int_\Omega   (f- v )\: p\: dx -\lambda   \int _\Omega     \nabla \cdot   V   \:      \int_0^{p} \beta_0(s) \: ds \:dx   .    	
		\end{array}
	\end{equation}  
	On the other, since $v$ is a weak solution of \eqref{st}, one sees easily that 
	\begin{equation}
		\lambda \int _\Omega\vert \nabla p\vert^2\: dx =  \int _\Omega  (f- v )\: p\: dx -\lambda    \int _\Omega     \nabla \cdot   V   \:     \int_0^{p} \beta_0(s)\: ds \:  dx  \: ; 
	\end{equation}  
	which implies that  $ 	\lim_{\delta\to 0}  
	\int_\Omega \vert \nabla p_\delta\vert^2\: dx=   \int_\Omega \vert \nabla p\vert^2\: dx.  $	Combing this with the weak convergence of $\nabla p_\delta,$ we deduce  the strong convergence \eqref{strongpeps}. 
\end{proof}

%Remark 
\begin{remark}
	One sees in the  proof that the results of Lemma \ref{lconveps}  remain to be true if one replace $f$ in \eqref{pstbeta} by a sequence of $f_\delta\in L^2(\Omega)$ such that 
	$$f_\delta \to f \quad \hbox{ in }L^2(\Omega), \quad \hbox{ as }\delta \to0.  $$
\end{remark}

%%%%%%%%%%%%%%%%%
\begin{proof}[\textbf{Proof of Proposition \ref{texistm}}]  
	The proof follows by Lemma \ref{lconveps}.  Moreover, the estimates hold to be true by letting $\delta\to 0,$    in the    estimate \eqref{lqstat} and \eqref{lmst} for $v_\delta$ and $p_\delta.$ 
\end{proof}

 \bigskip 
For the proof of $BV-$estimates of Theorem  \eqref{tbvlocuevol},we  now proceed to prove the following result for the {stationary problem}

%Theorem
\begin{theorem}  \label{tbvm}
Assume that $V\in  W^{1,\infty}(\Omega)^N ,$    $\nabla \cdot V\in  W^{1,2 }_{loc}(\Omega)$
and  satisfies \eqref{HypBV}. Let $v$ be the solution of \eqref{st} given by Lemma \ref{lconveps}.  Then, for any 
$$0<\lambda < \lambda_1:=  \left(\sum_{i,k} \Vert \partial_{x_i} V_k \Vert_{L^\infty(\Omega)}\right)^{-1},$$   
  $v\in BV_{loc}(\Omega)$, and   we have 
	\begin{equation} \label{bvstat} 
		\begin{array}{c}
			(1-\lambda \: \sum_{i,k} \Vert \partial_{x_i} V_k \Vert_{L^\infty(\Omega)}) 	\sum_{i=1}^N \int _\Omega    \omega_h  \:  {d\: \vert \partial_{x_i}v\vert}    \leq  \lambda	\sum_{i=1}^N    \int_\Omega  (\Delta \omega_h)^+  \:  \vert \partial_{x_i} p \vert   \: dx   +  	\sum_{i=1}^N \int_\Omega        \omega_h \: {d\:  \vert \partial_{x_i} f\vert}      \\  + \lambda   \sum_{i=1}^N   	  \int_\Omega         \omega_h \:  \vert v\vert \:   \vert  \partial_{x_i}   ( \nabla \cdot  V   )\vert     \: dx     ,
		\end{array}
	\end{equation} 
	for any   $0\leq \omega_h\in \mathcal H^2(\Omega_h)$   compactly supported in $\Omega,$ such that $\omega_h \equiv 1$ in  $    \Omega_h$ and 
	\begin{equation}  \label{HypsupportV2}  
		\	\int_{\Omega\setminus \Omega_h}  \varphi\:  V\cdot \nabla \omega_h  \: dx \leq  0,\quad \hbox{ for any }0\leq \varphi\in L^2(\Omega).
	\end{equation} 
\end{theorem}
%%%%%%%%% 

%Remark
\begin{remark} 
	Under the assumption $(T'3)$,  there exists at least one $w_h$ satisfying the  assumptions of Theorem  \ref{tbvm}. Indeed,   It is enough to take $\omega_h(x)=\eta_h(d(.,\Gamma)),$ for any $x\in \Omega,$ where $\eta_h\:  :\:   [0,\infty)\to \RR^+$ is  a  nondecreasing $\C^2$-function {on} $(0,\infty)$ such that   $\eta_h \equiv 1$  in $[	h,\infty)$. In this case, 
	$$\nabla \omega_h =\eta_h'(d(.,\Gamma))\: \nabla d(.,\Gamma), $$ 
	so that,   for any  $0\leq \varphi\in L^2(\Omega),$ we have 
	\begin{eqnarray*}   
		\	\int_{\Omega\setminus \Omega_h}  \varphi\:  V\cdot \nabla \omega_h  \: dx &=& \int_{\Omega\setminus \Omega_h} \eta_h'(d(.,\Gamma))\:   \varphi\:  V\cdot \nabla d(.,\Gamma)   \: dx  \leq   0,
	\end{eqnarray*}
	where we use \eqref{HypBV} and the fact that $ \nabla d(.,\Gamma) =-\nu(\pi(x))$ a.e. in $\Omega\setminus \Omega_h.$ 
	This function may be {defined}  
	\begin{equation} 
		\eta_h(r)= \left\{  \begin{array}{ll}
			0\quad &\hbox{ if } 0\leq r\leq c_1h\\   
			e^{\frac{-C_h}{r^2-c_1^2h^2 }}\quad &\hbox{ if } c_1 h\leq r\leq  c_2h\\   
			1-	e^{\frac{-C_h}{h^2 -r^2}} \quad &\hbox{ if } c_2h\leq r\leq h\\   
			1\quad &\hbox{ if }  h \leq r 	\end{array}  \right.
	\end{equation} 
	with $0<c_1<c_2<1$ and $C_h>0$ given such that $2c_2^2-c_1^2=1$ and $e^{\frac{-C_h}{M_h}} = 1-e^{\frac{-C_h}{M_h}},$ 
	where $M_h:= (c_2^2-c_1^2)h^2 =  (1-c_2^2)h^2.$   For instance one can take $c_1=1/2$ and $c_2=  \sqrt{5}/(2\sqrt{2}).$  See that 
	\begin{equation} 
		\eta_h'(r)= \left\{  \begin{array}{ll}
			0\quad &\hbox{ if } 0\leq r\leq c_1h\\   
			\frac{2rC_h}{(r^2-c_1^2h^2)^2}e^{\frac{-C_h}{r^2-c_1^2h^2 }}\quad &\hbox{ if } c_1 h\leq r\leq  c_2h\\   
			\frac{2rC_h}{(h^2-r^2)^2} 	e^{\frac{-C_h}{h^2 -r^2}} \quad &\hbox{ if } c_2h\leq r\leq h\\   
			0\quad &\hbox{ if }  h \leq r 	\end{array}  \right.
	\end{equation} 
	is continuous and {differentiable} at least on $\RR\setminus \{ c_2h\}.$ Thus $\eta_h\in H^2(\Omega).$  
\end{remark}

To prove Theorem \ref{tbvm}  we use again the regularized problem  \eqref{pstbeta} and we  let $\eps\to 0.$  To begin with, we prove first the following lemma concerning any weak solutions of the general problem 
\begin{equation}\label{eqalpha}	  
	\displaystyle  v -  \Delta \beta^{-1} (v)+     \nabla \cdot (v  \: V)=f  \hbox{ in } \Omega, 
\end{equation}
$\beta$  is a given  nondecreasing function assumed to be regular (at least $\C^2$).

%Lemma 		
\begin{lemma} \label{lbvreg}
Under the assumption of Theorem \ref{tbvm}, we assume  moreover that    $f\in W^{1,2}_{loc}(\Omega)$. Let us consider   $v\in H^1_{loc}(\Omega)\cap L^\infty_{loc}(\Omega)$ satisfying   \eqref{eqalpha} in $\D'(\Omega).$  Then,   for each $i=1,..N,$ we have    
	\begin{equation}\label{bvestreg} 
		\begin{array}{c}   
			\vert \partial_{x_i}	 v \vert - 
			\sum_{k=1}^N  \vert    \partial_{x_k} v\vert   \:    \sum_{k=1}^N  \vert       \partial_{x_i} V_k  \vert  -    \Delta 	\vert \partial_{x_i}  \beta^{-1} (v)  \vert  +      \nabla \cdot( \vert \partial_{x_i}  v\vert \: V)   \\   \leq \vert 	\partial_{x_i} f \vert 
			+   \vert v\vert \:\vert  \partial_{x_i}(\nabla \cdot  V) \vert  \quad \hbox{ in } \D'(\Omega).
		\end{array} 
	\end{equation}
\end{lemma}
\begin{proof}  Set $p:=\beta^{-1}(v).$  Thanks to   \eqref{eqalpha} and the regularity of $f$ and $V,$ it is not difficult to see that   $v,\: p\in H^2_{loc}(\Omega)\cap L^\infty_{loc}(\Omega),$   and for  each $i=1,...N,$  the partial derivatives $\partial_{x_i}v$ and $\partial_{x_i} p$ satisfy the following equation  
	\begin{equation}\label{bvint1}
		\partial_{x_i}	 v -  \Delta 	\partial_{x_i} p + 	    \nabla \cdot(\partial_{x_i}v\: V  )=	\partial_{x_i} f -  ( \nabla v\cdot   \partial_{x_i} V  +  v\: \partial_{x_i}  (\nabla \cdot  V)),\quad \hbox{ in } \D'(\Omega).
	\end{equation}  
	By density, we can take  $\xi H_\sigma(\partial_{x_i}v) $ as a test function in \eqref{bvint1} where $\xi \in H^2(\Omega)$  is    compactly supported  in $\Omega$ and  $H_\sigma$   is given by \eqref{Hsigma}. We obtain 
	\begin{equation}\label{bvint2}
		\begin{array}{cc} 	\int_\Omega  \Big( 	\partial_{x_i}	 v \:  \xi H_\sigma(\partial_{x_i}v)  +\nabla 	\partial_{x_i} p   \cdot \nabla (\xi H_\sigma(\partial_{x_i}v){)} \Big)  \: dx - \int_\Omega  	 \partial_{x_i}  v\: V \cdot \nabla (\xi H_\sigma (\partial_{x_i}v)) \: dx  \\  \\      =	\int_\Omega  \partial_{x_i} f\: \xi H_\sigma (\partial_{x_i}v) \: dx -   \int_\Omega     ( \nabla v\cdot   \partial_{x_i} V  +  v\: \partial_{x_i}  (\nabla \cdot  V) ) \:  \xi H_\sigma (\partial_{x_i}v)\: dx  . 
	\end{array} 	\end{equation} 
	To pass to the limit as $\sigma\to 0,$  we see first  that
	\begin{equation} \label{triv}
		H_\sigma '(\partial_{x_i}v)\: \partial_{x_i}v= \frac{1}{\sigma} \: \partial_{x_i}v  \: \chi_{[\vert \partial_{x_i} {v}\vert \leq \sigma ]}\:  \to 0 ,\quad \hbox{ in }L^\infty(\Omega)\hbox{-weak}^*.
	\end{equation}
	So,   the last term of the first part of \eqref{bvint2} satisfies  
	\begin{eqnarray}
		\lim_{\sigma \to 0}	 \int_\Omega  	 \partial_{x_i}  v\: V \cdot \nabla (\xi H_\sigma (\partial_{x_i}v)) \: dx &=&     \int_\Omega  	 \vert \partial_{x_i}  v\vert \: V \cdot \nabla  \xi  \: dx +  \lim_{\sigma  \to 0}	 \int_\Omega   \partial_{x_i}v  \: \nabla   \partial_{x_i}v  \cdot  V   H_\sigma '(\partial_{x_i}v) \: \xi  \: dx \\  
		&=&   \int_\Omega  	 \vert \partial_{x_i}  v\vert \: V \cdot \nabla  \xi.
	\end{eqnarray}
	On the other hand, we see that 
	\begin{eqnarray}
		\int _\Omega \nabla 	\partial_{x_i} p   \cdot \nabla (\xi H_\sigma (\partial_{x_i}v) ) \: dx  &=&\int_\Omega   H_\sigma (\partial_{x_i}v)  	\nabla 	\partial_{x_i} p   \cdot \nabla \xi   \: dx  + \int_\Omega  \xi\: 	\nabla 	\partial_{x_i} p   \cdot \nabla  H_\sigma (\partial_{x_i}v) \: dx .  
	\end{eqnarray}
	Since $\sign_0 (\partial_{x_i}v)=\sign_0 (\partial_{x_i}  p),$   the first term satisfies
	\begin{equation}
		\lim_{\sigma\to 0} \int  H_\sigma (\partial_{x_i}v)  	\nabla 	\partial_{x_i} p   \cdot \nabla \xi   \: dx  =-   \int   \vert \partial_{x_i} p\vert  \: \Delta  \xi   \: dx.  
	\end{equation}  
	As the second term, we have 
	\begin{eqnarray}
		\lim_{\sigma\to 0} 	\int_\Omega  \xi\: 	\nabla 	\partial_{x_i} p   \cdot \nabla  H_\sigma (\partial_{x_i}v) \: dx &=&   \lim_{\sigma\to 0} 	\int_\Omega  \xi\: H_\sigma '(\partial_{x_i}v)\: 	\nabla 	\partial_{x_i} p   \cdot \nabla \partial_{x_i}v    \: dx \\  
		&=&   \lim_{\sigma\to 0} 	\int_\Omega  \xi\: H_\sigma '(\partial_{x_i}v)\: 	\nabla   (\beta'(v)\partial_{x_i}v)   \cdot \nabla \partial_{x_i}v    \: dx\\  
		&=&   \lim_{\sigma\to 0} 	\int_\Omega  \xi\: H_\sigma '(\partial_{x_i}v)\: 	\beta'(v) \: \Vert \nabla  \partial_{x_i}v\Vert^2   \: dx   \\   &  & +    \lim_{\sigma\to 0} 	\int_\Omega  \xi\: H_\sigma '(\partial_{x_i}v)\: \partial_{x_i}v \: \beta''(v) 	\nabla  v    \cdot \nabla \partial_{x_i}v    \: dx  \\  
		&\geq &  \lim_{\sigma\to 0} 	\int_\Omega  \xi\: H_\sigma '(\partial_{x_i}v)\: \partial_{x_i}v \: \beta''(v) 	\nabla  v    \cdot \nabla \partial_{x_i}v    \: dx \\   &= &0, 
	\end{eqnarray}
	where we use again \eqref{triv}.  
	So,  letting   $\sigma\to 0$ in \eqref{bvint2} and  	using  again the fact that   $\sign_0 (\partial_{x_i}v)=\sign_0 (\partial_{x_i}  p),$        
	we get    
	$$ \begin{array}{c}   
		\vert \partial_{x_i}	 v \vert -  \Delta 	\vert \partial_{x_i} p \vert  +    \nabla \cdot( \vert \partial_{x_i}  v\vert \: V)     \leq \sign_0 (\partial_{x_i}  v) 	\partial_{x_i} f    -           ( \nabla v\cdot   \partial_{x_i} V  \\   \\   +  v\: \partial_{x_i}(\nabla \cdot  V) ) \: \sign_0 (\partial_{x_i}v) \quad \hbox{ in } \D'(\Omega).
	\end{array}  $$  
	At last, using the fact that 
	$\vert \nabla v\cdot   \partial_{x_i} V\vert \leq   \sum_{k}  \vert    \partial_{x_k} v\vert   \:    \sum_{k}  \vert       \partial_{x_i} V_k  \vert  ,$ 
	the result of the lemma follows.   
\end{proof}

%%%%%%%%%%%% 
\begin{proof}[\textbf{Proof of Theorem \ref{tbvm}}] 	Under the assumptions of Theorem  \ref{tbvm}, for any $\epsilon >0,$ let us consider $f_\eps$  a  regularization   of $f$  satisfying $f_\eps \to f$ in $L^1(\Omega)$ and 
	\begin{equation}
		\int_\Omega \xi\: 	\vert \partial_{x_i} f_\eps\vert \: dx \to  \int_\Omega \xi\: d\: 	\vert \partial_{x_i} f\vert  \quad   \hbox{ for any } \xi\in \C_c(\Omega)  \hbox{ and }  i=1,...N.
	\end{equation} 	
	Thanks to Lemma \ref{lexistreg},  we consider $v_\eps$ be the solution of the problem \eqref{st}, where we replace  $f$ by the regularization  $f_\eps$.    Applying   Lemma \ref{lbvreg} by replacing  $V$ by $\lambda V$ and $\beta^{-1}$ by $\lambda\: \beta_\eps^{-1}$,   we obtain  
	\begin{equation} 
		\begin{array}{c}
			\int _\Omega      \vert \partial_{x_i}v_\eps \vert     \: \xi\: dx     -  \lambda     \int_\Omega     \sum_{k=1}^N   \vert       \partial_{x_i} V_k  \vert  \:       \sum_{k=1}^N  \vert \partial_{x_k}v_\eps\vert    \: \xi\: dx   
			\leq  \lambda   \sum_{k=1}^N  \int _\Omega 	\vert \partial_{x_i} p_\eps \vert  \: (\Delta \xi)^+\:    dx + 
			\int_\Omega  \vert 	\partial_{x_i} f_\eps \vert \: \xi \: dx   \\      +   	 \lambda  \int _\Omega    \vert    
			v_\eps\vert \:  \vert \partial_{x_i} ( \nabla \cdot     V  ) \vert 	 \: \xi \: dx    + \lambda   \:    \int _\Omega     \vert \partial_{x_i}v_\eps   \vert\:    V\cdot \nabla \: \xi\: dx   	\quad \hbox{ for any }i=1,...N \hbox{ and }0\leq \xi\in \D(\Omega).
		\end{array}
	\end{equation}   
	By density, we can take $\xi=\omega_h$ as given in the theorem, so that the last term   is nonpositive, and we have 
	\begin{equation} 
		\begin{array}{c}
			\int _\Omega      \vert \partial_{x_i}v \vert     \: \omega_h\: dx     -  \lambda        \:    \sum_{k}  \Vert       \partial_{x_i} V_k  \Vert_{L^\infty(\Omega)}   \:    \int _\Omega  \sum_{k} \vert \partial_{x_k}v\vert    \: \omega_h\: dx      	\leq  \lambda  \sum_{k}   \int _\Omega 	\vert \partial_{x_i} p \vert  \: (\Delta  \omega_h)^+\:    dx    \\   + 
			\int_\Omega  \vert 	\partial_{x_i} f  \vert \: \omega_h \: dx   +  	 \lambda  \int_\Omega     \vert    
			v\: \vert \:  \vert \partial_{x_i} ( \nabla \cdot     V  ) \vert 	 \: \omega_h  \: dx   	,\quad \hbox{ for any }i=1,...N.
		\end{array}
	\end{equation}   
	Summing up,  for $i=1,...N,$  and using the definition of $\lambda_1,$ we deduce that  
	\begin{equation} 
		\begin{array}{c}
			\sum_{i} 		\int _\Omega      \vert \partial_{x_i} v_\eps \vert     \: \omega_h\: dx     -  \lambda \lambda_1   \sum_{k}   \int_\Omega   \vert \partial_{x_k}v_\eps\vert    \: \omega_h\: dx   	\leq  \lambda  \sum_{i}   \int _\Omega  	(\Delta \omega_h)^+\: \vert \partial_{x_i} p_\eps \vert  \: dx    \\    + 
			\int_\Omega \sum_{i}  \vert 	\partial_{x_i} f_\eps  \vert \: \omega_h \: dx    + \lambda \:  \int_\Omega    \vert v_\eps\vert \sum_{i}      \vert \partial_{x_i} (\nabla \cdot   V  )  \vert  \: \omega_h \: dx  ,
		\end{array} 
	\end{equation}   
	and then the corresponding property \eqref{bvstat} follows for $v_\eps.$ 
	Thanks to  \eqref{lqstat} and \eqref{lmst}, we know that $v_\eps$ and $\partial_{x_i}p_\eps$ are bounded in $L^2(\Omega).$ This implies that, for any $\omega\subset\!\subset \Omega,$  $\sum_{i} 		\int_\omega      \vert \partial_{x_i} v_\eps \vert     \:   dx$ is bounded.  So,  $v_\eps$ is bounded in $BV_{loc}(\Omega).$  Combining this with the $L^1-$bound \eqref{lqstat}, it implies  in particular,  taking a subsequence if necessary,     the convergence in \eqref{weakueps} holds to be true also in $L^1(\Omega)$ and   then $v\in BV_{loc}(\Omega).$   At last, letting $\eps\to 0$ in \eqref{bvestreg} and, using moreover \eqref{strongpeps} and the lower semi-continuity of variation measures   $\vert \partial_{x_i} v_\eps\vert$, we deduce \eqref{bvstat}.
\end{proof}

\begin{remark}\label{Rbvcond}
	See that we use the condition $(T'3)$  for the proof of  local $BV$-estimate through $w_h$ we introduce  in Theorem \ref{tbvm}.  Indeed,  local $BV$ estimates follows  from Lemma \ref{lbvreg}  once  the term $  \lambda   \:    \int_\Omega      \vert \partial_{x_i}v_\eps   \vert\:    V\cdot \nabla \: \xi\: dx$ nonpositive. 
	Clearly, the construction of  $\omega_h$ by using $d(.,\partial \Omega)$ is  basically connected to the condition $(T'3)$.  Otherwise, this condition could be replaced definitely by the assumption 
	\begin{equation}
		\exists\:  0\leq \omega_h\in \mathcal H^2(\Omega_h),\: \hbox{support}(\omega_h)\subset \!\subset \Omega, 
		\omega_{h_{/\Omega_h}}\equiv 1\hbox{ and } \eqref{HypsupportV2}\hbox{ is fulfilled.}
	\end{equation}
\end{remark}

%%%%%
 \section{Proofs for the evolution problem}
\setcounter{equation}{0}

 Thanks to Proposition \ref{texistm}, one sees that for $\eps$ small enough, the $\eps$-approximate solution, $u_\eps$, as given in \eqref{epsapprox} is well defined. To prove Theorem \ref{texistevolm}, we need to pass to the limit by letting $\eps\to 0.$

%Lemma
\begin{lemma}\label{lemmaeps} Let $(u_\eps,p_\eps)$ a solution of problem (\ref{weakeps0}), then
	
	\begin{enumerate}
	
		\item For any $q\in [1,\infty],$ we have  
		  \begin{equation} \label{lquepsevol} 
			\Vert u_\eps(t)\Vert_{L^q(\Omega)}  \leq  \M_q^\eps (t)\quad  \hbox{ for any }t\geq 0,   
		\end{equation}     
 \begin{equation} 
 	\M_q^\eps(t) := 
 		   \left( \Vert u_0 \Vert_{L^q(\Omega)} + \int_0^t   \Vert f_\eps (t)\Vert_{L^q(\Omega)} \: dt  \right) e^{  (q-1)\:  \int_0^t  \Vert (\nabla \cdot V)^-\Vert_{L^\infty(\Omega)}  }.
				\end{equation}   
		\item  For each $\eps >0,$  we have  
			\begin{equation}\label{lmuepsevol}
			\begin{array}{c} 
			  \int_\Omega j( u_\eps(t)) +  \int_0^t\!\!  \int_\Omega \vert \nabla p_\eps\vert^2  \leq   \int_0^t\!\!     \int_\Omega f_\eps\:p_\eps \: dx +   	   \int_0^t\!\!   \int \left(  \nabla \cdot  V\right)^- \:    p_\eps \: u_\eps     \: dx +   \int_\Omega j( u_{0}).  
			\end{array} 
		\end{equation} 	  
		In particular, we have 
		\begin{equation}\label{estim2}  	
			\frac{d}{dt} \int_\Omega j( u_\eps) \: dx+  \int_\Omega \vert \nabla p_\eps\vert^2\: dx    \leq    C (N,\Omega)   \left(  	\int_\Omega \vert f_\eps\vert^2\: dx +	 \Vert (\nabla \cdot V)^-  \Vert_{L^\infty(\Omega)}\: (\M_2^\eps(t))^2      \right ) \quad \hbox{ in }\D'(0,T).
		\end{equation}

	 \item If $u_0\geq 0$ and $f\geq 0,$ then ${u_\eps} \geq 0.$	
	
	\end{enumerate}	
\end{lemma}
\begin{proof}
	Thanks to Proposition \ref{texistm}, the  sequence $  (u_i)_{i=1,...n}$ of solutions of \eqref{sti} is well defined in $L^2(\Omega) $  and satisfies  
	\begin{equation}\label{stwsi}
		\displaystyle  \int_\Omega u_{i+1} \:\xi+   \eps \: \int_\Omega  \nabla p_{i+1}  \cdot  \nabla\xi  -  \eps \:
		\int_\Omega  u_{i+1}  \:  V_i\cdot \nabla  \xi    =  \int_\Omega u_{i}   \:\xi+  \int_{i\eps  }^{(i+1)\eps  } \int_\Omega f_i\: \xi \quad \hbox{ for }i={0,...,n-1} .
	\end{equation} 
	Thanks to \eqref{lqstat}, we have
\begin{eqnarray}
	 \Vert u_i\Vert_{L^q(\Omega)} &\leq&  	 \Vert u_{i-1}\Vert_{L^q(\Omega)}  + \eps\: \Vert f_{i}\Vert_{L^q(\Omega)}   +   \eps\: (q-1) \:   \Vert (\nabla \cdot V_i)^-\Vert_{L^\infty(\Omega)}  \Vert u_{i}\Vert_{L^q(\Omega)} \qquad \mbox{ and }\;  u_i \geq 0.
  \end{eqnarray}
By induction, this implies that,  for any $t\in [0,T),$ we have  
  \begin{equation}  
  \Vert u_\eps(t)\Vert_{L^q(\Omega)}    \leq     \Vert u_0 \Vert_{L^q(\Omega)}    + \int_0^T \Vert f_\eps(t)\Vert_{L^q(\Omega)} \: dt    +  (q-1) \:        \int_0^T \Vert (\nabla \cdot V_i)^-\Vert_{L^\infty(\Omega)}   \Vert u_\eps(t)\Vert_{L^q(\Omega)}\: dt   .
	\end{equation}  
Using Gronwall Lemma, we deduce  \eqref{lquepsevol}.  
	Using  the fact that   
	\begin{equation}
		\left(u_i-u_{i-1} \right) p_i \geq  j (u_i) -j(u_{i-1})
	\end{equation}
and 
\begin{equation}
	\int u_i  \:  V_i \cdot \nabla p_i \leq \int \int\left(\nabla \cdot V_i\right)^- p_i\: u_i , 
\end{equation}	we get  
	\begin{equation} 
		    \int_\Omega j(u_i)+ \eps \int_\Omega \vert \nabla p_i\vert^2  \leq    \eps  \int_\Omega f_i\:p_i \: dx +   	\eps  \:  \int \left(  \nabla \cdot  V_i\right)^- \:    p_i \: u_i     \: dx  +    \int_\Omega j( u_{i-1}).
	\end{equation} 
 Summing this identity for $i=1,....,$  and using the definition of $u_\eps$, $p_\eps$ and $f_\eps,$ we get \eqref{lmuepsevol}.  At last, it is not difficult to see that \eqref{estim2} is a simple consequence of  \eqref{lquepsevol} and \eqref{lmuepsevol} combined with Poincaré and Young  inequalities.  
\end{proof}

\bigskip 
\begin{proof}[\textbf{Proof of Theorem \ref{texistevolm}}]  
	See that \eqref{convueps} and \eqref{convpeps}  is a consequence of 
	\eqref{lquepsevol} with $q=2$ and  \eqref{estim2} respectively.   	As to $\tilde u_\eps,$ we see that  \begin{equation}\label{tildeeps-rel}
		\tilde  u_\eps(t)-  u_\eps(t) =  (t-t_i)\: \partial_t \tilde  u_\eps(t)     ,\quad \hbox{ for any }t\in ]t_{i-1},t_{i}],\: i=1,... n.   
	\end{equation}
	Since $u_\eps,$ $p_\eps$ and $f_\eps$ are bounded in $L^2(Q),$ $L^2(0,T;H^1_{\Gamma_D}(\Omega)) $ and $L^2(Q),$ respectively, we see that $\partial_t \tilde  u_\eps$ is bounded in $L^{2} (0,T;H^{-1}_{\Gamma_D}(\Omega) ).$ Combining this with the fact that $\tilde u_\eps$ is bounded in $L^2(Q),$ we deduce in turn \eqref{convtildeueps}.  Then, letting $\eps\to 0$ in \eqref{weakeps0}, we deduce that the couple $(u,p)$ satisfies the weak formulation  \eqref{evolweak}.  To prove that $u\in \beta(p)$ a.e. in $Q,$   we use the   very weak Aubin's result   (cf.  \cite{ACM} or \cite{Moussa}).    We combine  \eqref{convueps}, \eqref{convpeps} with the fact that    $\partial_t \tilde  u_\eps$ is bounded in $L^{2} (0,T;H^{-1}_{\Gamma_D}(\Omega) ),$   to deduce first that, as $\tau\to 0,$ we have  
	\begin{equation}
		\int_0^T\!\!\int_\Omega    u_\eps \:  p_\eps \: \varphi\: dxdt \to 	\int_0^T\!\!\int_\Omega u \: p  \: \varphi\: dxdt  \quad \hbox{ for any }\varphi\in \D ( Q). 
	\end{equation}

	Then, using the fact that $u_\eps \in \beta(p_\eps)$ a.e. in $Q$, we conclude that
	$$
	u\in \beta(p)\quad\text{a.e. in }Q,
	$$
	by the Minty trick (closure of maximal monotone graphs): from $u_\eps\rightharpoonup u$ in $L^2(Q)$,
	$p_\eps\rightharpoonup p$ in $L^2(Q)$ (or in $L^2(0,T;H^1)$), and the product convergence
	$\int_Q u_\eps p_\eps\varphi \to \int_Q up\,\varphi$ for all $\varphi\in\mathcal D(Q)$, one deduces that the limit pair
	$(p,u)$ belongs to the graph of $\beta$.
	
\end{proof}

%Lemma
%%%%%%%% 
\begin{lemma}\label{Lbvestim}
	Under the assumption of Theorem \ref{tbvm},    for any $\omega\subset\! \subset\Omega,$ there   exists   $0\leq \xi \in \mathcal C^2_c(\Omega)$ such that   $\xi\chi_{\omega}\equiv 1,$ and  
\begin{equation}\label{bvuepsevol} 
		\begin{array}{c}
	  	\sum_{j=1}^N 	\int        \xi\: d \vert \partial_{x_j} u_\eps(t) \vert     \leq         e^{\lambda_VT}  \sum_{j=1}^N   \left(       \int  (\Delta \xi)^+ \: 	{\vert \partial_{x_j} p_\eps \vert}  \: dx   +      \int \xi\:  \vert u_\eps\vert\: {\vert \partial_{x_j}   ( \nabla \cdot   V   ) \vert} \: dx  \right. \\  \left.  +	   \int \xi\: d\vert 	\partial_{x_j} f_\eps(t) \vert    +		 \int \xi\: d \vert 	\partial_{x_j} u_{0}  \vert   \right)   ,  
		\end{array}	
	\end{equation}   
\end{lemma}
\begin{proof} We fix  $\xi  $ as given by Theorem \ref{tbvm}.   Let us prove that 
		\begin{equation}\label{bvuepsevol0} 
		\begin{array}{c}
			\sum_{j=1}^N \left( 	\int        \xi\: d \vert \partial_{x_j} u_\eps(t) \vert    -   \int_{0}^{T}\!\! \!   	\int        \xi\: d \vert \partial_{x_j} u_\eps(t) \vert  \: dt  \right) \leq   	\sum_{j=1}^N     \int_{0}^{T} \!\! \! \left(     \int  (\Delta \xi)^+\: 	\vert \partial_{x_j} p_\eps \vert  \: dx +       \int \xi\:   \vert u_\eps\vert\:  \vert \partial_{x_j}   ( \nabla \cdot   V   ) \vert {\: dx}   \right.  \\   +	\left.     \int \xi\: d\vert 	\partial_{x_j} f_\eps(t) \vert    +		  \int \xi\: d \vert 	\partial_{x_j} u_{0}  \vert \right)     ,  
		\end{array}	
	\end{equation}  
Recall that $u_i$ is a solution of 
$$u_i-\eps\: \Delta p_i+\eps\: \nabla \cdot (u_i\: V_i)=\eps\: f_i,\quad \hbox{ with } u_i\in \beta(p_i) \hbox{ for each  }i=1,...,n.$$
 	Thanks to Theorem \ref{tbvm}, we have 
	 	\begin{equation} 
		\begin{array}{c}
			(1-\eps \lambda_V ) 	\sum_{k=1}^N \int     \xi  \:  d\vert \partial_{x_k} u_i\vert   \leq  \eps  \sum_{k=1}^N      \int  (\Delta \xi)^+\:   \vert \partial_{x_k} p_i \vert   \: dx      + \eps   \sum_{k=1}^N   	  \int        \xi \:  \vert u_i\vert \:   \vert  \partial_{x_k}   ( \nabla \cdot  V_i   )\vert     \: dx    \\   + \eps 	\sum_{k=1}^N \int       \xi \: d \vert \partial_{x_k} f_i\vert     + \sum_{k=1}^N 	 \int     \xi  \:  d\vert \partial_{x_k} u_{i-1}\vert  ,    .
		\end{array}
	\end{equation} 
where we use the fact that 
$$\vert \partial_{x_k} (\eps\: f_i+u_{i-1})\vert  \leq \eps\: \vert \partial_{x_k}  f_i \vert  + \vert \partial_{x_k}  u_{i-1}\vert .$$
Using the definition of $f_i,$ we get 
	 	\begin{equation}  
	\begin{array}{c}
		(1-\eps \lambda_V ) 	\sum_{k=1}^N \int     \xi  \:  d\vert \partial_{x_k} u_i\vert   \leq  \eps   \sum_{k=1}^N      \int  (\Delta \xi)^+\:  \vert \partial_{x_k} p_i \vert   \: dx  + \eps   \sum_{k=1}^N   	  \int        \xi \:  \vert u_i\vert \:   \vert  \partial_{x_k}   ( \nabla \cdot  V_i   )\vert     \: dx    \\    +  	\sum_{k=1}^N \int_{t_i}^{t_{i+1}}\!\! \! \int       \xi \: d \vert \partial_{x_k} f \vert          + \sum_{k=1}^N 	 \int     \xi  \:  d\vert \partial_{x_k} u_{i-1}\vert      .
	\end{array}
\end{equation} 
Iterating from $i=1$ to $n,$ this implies that  for any $t\in [0,T),$ we have 
	\begin{equation}  
	\begin{array}{c}
		\sum_{k=1}^N \int     \xi  \:  d\vert \partial_{x_k} u_\eps(t)\vert  - 	 \lambda_V  	\sum_{k=1}^N  \int_0^T\!\!\!   \int     \xi  \:  d\vert \partial_{x_k} u_\eps(t)\vert \: dt  \leq      \sum_{k=1}^N     \int_0^T\!\!\!  \int  (\Delta \xi)^+\:  \vert \partial_{x_k} p_\eps \vert   \: dxdt   \\    +     \sum_{k=1}^N   	 \int_0^T\!\!\!   \int        \xi \:  \vert u_\eps \vert \:   \vert  \partial_{x_k}   ( \nabla \cdot  V_\eps    )\vert     \: dxdt   +  	\sum_{k=1}^N \int_0^T\!\!\!  \int       \xi \: d \vert \partial_{x_k} f \vert   dt       + \sum_{k=1}^N 	 \int     \xi  \:  d\vert \partial_{x_k} u_0\vert      .
	\end{array}
\end{equation} 
Thus the result of the lemma follows by applying Gronwall's inequality. 
\end{proof}

\begin{proof}[\textbf{Proof of Theorem \ref{tbvlocuevol}}] 
For a fixed $t\in (0,T)$ and $\omega\subset\! \subset\Omega,$ we know by Lemma \ref{Lbvestim}  that  there   exists   $0\leq \xi \in \mathcal C^2_c(\Omega)$ such that   $\xi\chi_{\omega}\equiv 1$ and 

{
\begin{equation} 
	\begin{array}{c}
		\sum_{j=1}^N 	\int        \xi\: d \vert \partial_{x_j} u_\eps(t) \vert     \leq         e^{\lambda_VT}  \sum_{j=1}^N   \left(       \int  (\Delta \xi)^+ \: 	\vert \partial_{x_j} p_\eps \vert  \: dx   +      \int \xi\: 
		\abs{u_\eps} \vert \partial_{x_j}  ( \nabla \cdot   V   ) \vert {\: dx}  \right. \\
		\left.  +	   \int \xi\: d\vert 	\partial_{x_j} f_\eps(t) \vert    +		 \int \xi\: d \vert 	\partial_{x_j} u_{0}  \vert   \right)   ,  
	\end{array}	
\end{equation}  
}

Combining this with   \eqref{lquepsevol} and 	\eqref{estim2}, we can find  a constant $ C:= C(\omega,N,\Omega)$ depending only on $\omega$ and  Poincaré inequality, such that  
\begin{equation}\label{bvlocuevoleps} 
	\begin{array}{cc}
		\sum_{j=1}^N    \vert \partial_{x_j} u_\eps(t) \vert (\omega)	 \leq C\left( 1+ \int_0^T\!\! \int_\Omega \vert f_\eps\vert^2\: dx +	 \Vert (\nabla \cdot V)^-  \Vert_{L^\infty(\Omega)}\: (\M_2)^2     +  	\sum_{j=1}^N     \int  \vert  \partial_{x_j}(\nabla \cdot  V {)} \vert^2 \: dx     \right. \\ 
		\hspace*{1cm}	\left.     +	 \sum_{j=1}^N    	\int_0^T  \vert \partial_{x_j} f(t) \vert(\Omega)+  \sum_{j=1}^N    \vert \partial_{x_j} u_0 \vert(\Omega) \right),\quad \hbox{ for any }t\in (0,T).
	\end{array}
\end{equation}  
So, for any $\omega\subset\!\subset \Omega,$  $\sum_{i} 		\int_\omega      \vert \partial_{x_i} u_\eps \vert     \:   dx$ is bounded.  So,  $u_\eps(t)$ is bounded in $BV (\omega) $, and using     \eqref{convueps}, we deduce that    $u(t)\in BV (\omega).$   Moreover,    letting $\eps\to{0}$ in \eqref{bvlocuevoleps} and using  the lower semi-continuity of variation measures   $\vert \partial_{x_i} u_\eps(t)\vert$, we deduce  
\eqref{bvlocuevol} for the limit $u$.
 
\end{proof}

  	\section*{Acknowledgments } This work   was  supported by the CNRST of Morocco under the FINCOM program.  N. Igbida  is very grateful to the EST of Essaouira and its staff for their warm hospitality and the conducive work environment they provided.


\vspace*{5mm}
 \begin{thebibliography}{1}
 	
% 	\bibitem{AKY}
% 	D.~{\sc Alexander}, I.~{\sc Kim}, and Y.~{\sc Yao}.
% 	\newblock Quasi-static evolution and congested crowd transport.
% 	\newblock {\em {Nonlinearity}} \textbf{27} (2014), no.~4, 823--858.
% 	
% 	\bibitem{AltLu}
% 	H.~W. {\sc Alt} and H.~W. {\sc Luckhaus}.
% 	\newblock Quasilinear elliptic-parabolic differential equations.
% 	\newblock {\em Math. Z.} \textbf{183} (1983), 311--341.
% 	
% 	\bibitem{ACMa}
% 	L. {\sc Ambrosio}, G. {\sc Crippa} and S. {\sc Maniglia}.
% 	\newblock{Traces and fine properties of a BD class of vector fields and applications}.
% 	\newblock{Ann. Fac. Sci. Toulouse Math}, (6), 14(4):527-561, 2005.
	
 	\bibitem{ACM}
 	B.~{\sc Andreianov}, C.~{\sc Cances}, and A.~{\sc Moussa}.
 	\newblock A nonlinear time compactness result and applications to discretization of degenerate parabolic-elliptic PDEs.
 	\newblock {\em J. Func. Anal.} \textbf{273} (2017), no.~12, 3633--3670.
  
 	\bibitem{Ca}
 	J.~{\sc Carrillo}.
 	\newblock Entropy solutions for nonlinear degenerate problems.
 	\newblock {\em Arch. Rational Mech. Anal.} \textbf{147} (1999), 269--361.
 	
 	
 	%		
 	\bibitem{Donadello1}
 	G.~{\sc Crippa}, C.~{\sc Donadello}, and L.~{\sc Spinolo}.
 	\newblock Initial-boundary value problems for continuity equations with BV coefficients.
 	\newblock {\em Journal Math. Pure Appl.} \textbf{102} (2014), no.~1, 79--98.
 	
 	
 	
 	\bibitem{Donadello2}
 	G.~{\sc Crippa}, C.~{\sc Donadello}, and L.~{\sc Spinolo}.
 	\newblock A note on the initial-boundary value problem for continuity equations with rough coefficients.
 	\newblock In {\em Proceedings of HYP2012} (2012).
 	
 	
 	\bibitem{Noemi}
 	N.~{\sc David} and M.~{\sc Schmidtchen}.
 	\newblock On the incompressible limit for a tumour growth model incorporating convective effects.
 	\newblock {\em Communications on pure and applied Mathematics} \textbf{77} (2024), no.~5, 2613--2650.
 	
 	\bibitem{DeMeSan}
 	G.~{\sc De Philippis}, A.~{\sc M\'esz\'aros}, F.~{\sc Santambrogio}, and B.~{\sc Velichkov}.
 	\newblock BV Estimates in Optimal Transportation and Applications.
 	\newblock {\em Arch. Rational Mech. Anal.} \textbf{219} (2016), 829--860.
 
 	
 	\bibitem{DiMe}
 	S.~{\sc Di Marino} and A.~R. {\sc M\'esz\'aros}.
 	\newblock Uniqueness issues for evolution equations with density constraints.
 	\newblock {\em Mathematical Mod. Meth. App. Sci.} \textbf{26} (09), 1761--1783.
 	
 	
 	
 	\bibitem{IgCrowed}
 	H.~{\sc Ennaji}, N.~{\sc Igbida}, and G.~{\sc Jradi}.
 	\newblock Prediction-Correction Pedestrian Flow by Means of Minimum Flow Problem.
 	\newblock {\em Math. Models \& Methods Appl. Sci.} \textbf{34} (2024), no.~3, 385--416.
 	  
 	\bibitem{IgShaw}
 	N.~{\sc Igbida}.
 	\newblock $L^1$-Theory for Hele-Shaw flow with linear drift.
 	\newblock {\em Math. Models \& Methods Appl. Sci.} \textbf{33} (2023), no.~7, 1545--1576.
 	  
 	\bibitem{IgPME}
 	N.~{\sc Igbida}.
 	\newblock $L^1$-Theory for {Incompressible} limit of Porous Medium Equation with linear Drift.
 	\newblock {\em J. Differential Equations} \textbf{416} (2025), Part~2, 1015--1051.

 	\bibitem{IgMF}
 	N.~{\sc Igbida}.
 	\newblock Minimum Flow Steepest Descent Approach for Nonlinear PDE.
 	\newblock {\em https://arxiv.org/abs/2402.02134}.


 	  
 	\bibitem{MRS1}
 	B.~{\sc Maury}, A.~{\sc Roudneff-Chupin}, and F.~{\sc Santambrogio}.
 	\newblock A macroscopic crowd motion model of gradient flow type.
 	\newblock {\em Math. Models \& Methods Appl. Sci.} \textbf{20} (2010), no.~10, 1787--1821.


 	\bibitem{Moussa}
 	A.~{\sc Moussa}.
 	\newblock Some variants of the classical Aubin--Lions Lemma.
 	\newblock {\em J. of Evolution Equations} \textbf{16} (2016), 65--93.
 	
 	
 	\bibitem{Santambrogio15}
 	F.~{\sc Santambrogio}.
 	\newblock {\em Optimal Transport for Applied Mathematicians}.
 	\newblock Birk\"auser, 2015.
 	
 	
 	\bibitem{Vbook}
 	J.~L. {\sc V\'azquez}.
 	\newblock The porous medium equation: mathematical theory.
 	\newblock Oxford University Press, 2007.
 	
 	 
 	
 	
 \end{thebibliography}
\end{document}